\documentclass{amsart}
\usepackage{graphicx}
\usepackage{latexsym}
\usepackage{hyperref}
\usepackage{amsfonts}
\usepackage[all]{xy}
\usepackage{amssymb, mathrsfs, amsfonts, amsmath}
\usepackage{amsbsy}
\usepackage{amsfonts}
\newtheorem{definition}{Definition}
\newtheorem{lemma}{Lemma}

\newtheorem{theorem}{Theorem}

\numberwithin{equation}{section}

\begin{document}
\let\WriteBookmarks\relax
\def\floatpagepagefraction{1}
\def\textpagefraction{.001}



\title [mode = title]{Four-dimensional electrostatic system with harmonic (anti-)self-dual Weyl tensor}  

%

%

\author[1]{Róbson Lousa}

%
%



%
%
%
%




\begin{abstract}
We investigate four-dimensional electrostatic systems arising as spatial factors of static Einstein--Maxwell spacetimes with cosmological constant. Assuming that the electric field is everywhere collinear with the gradient of the lapse function, we prove that the harmonicity of one of the (anti-)self-dual components of the Weyl tensor imposes strong rigidity on the underlying geometry. More precisely, we show that the gradient of the lapse function is an eigenvector of the Ricci tensor and that the regular level sets of the lapse function are totally umbilic with constant mean curvature. As a consequence, the manifold is locally conformally flat and admits a local warped product structure with one-dimensional base and three-dimensional fiber of constant curvature.
\end{abstract}





\maketitle

 	\section{Introduction and main results}
 
 The Hodge star operator plays a central role in four-dimensional Riemannian geometry. On an oriented $4$-manifold, it acts as an involution on $\Lambda^2$, and this leads to the classical splitting of the Weyl tensor into its self-dual and anti-self-dual parts (see, for instance, \cite[Chapter 13]{besse}):
 \[
 W \;=\; W^+ \oplus W^-.
 \]
 We say that $W^+$ is \emph{harmonic} if its divergence vanishes, $\delta W^+ = 0$, and similarly for $W^-$. Conditions on $W$ or on its (anti-)self-dual components are a natural way to control the conformal geometry of the underlying manifold: for example, $W \equiv 0$ if and only if the manifold is locally conformally flat in dimension $n\ge4$.
 
 In dimension four there is extensive literature relating vanishing or divergence conditions on (parts of) $W$ to strong classification results. Let us mention, among others, the works \cite{barros2015critical, catino2012note,deng2015note, catino2017gradient, neto2016generalized}, where (anti-)self-dual and harmonic Weyl tensors are studied in the context of quasi-Einstein metrics, Ricci solitons, and related structures. In particular, every four-dimensional Einstein manifold has harmonic $W^+$ (see \cite[16.65]{besse}), which motivates the question of how far one can push similar conclusions for more general classes. See also \cite{catino-mastrolia-monticelli-punzo, derdzinski1983, gursky2000, lebrun2021,   tran2017} for further rigidity results in the presence of harmonic (self-dual) Weyl curvature.
 
 The present approach is inspired by the theory of integrability conditions, introduced by Cao and Chen \cite{cao} in the study of gradient Ricci solitons and subsequently developed in several geometric contexts. In particular, suitable covariant 3-tensors encoding the interaction between the Cotton and Weyl tensors have proved to be a powerful tool in deriving rigidity results. Our construction follows this philosophy in the setting of four-dimensional electrostatic systems with cosmological constant \cite{cao,catino-mastrolia-monticelli-punzo,catino2017gradient}.

 In a different direction, electrostatic systems arise naturally as spatial slices of static Einstein--Maxwell spacetimes with cosmological constant. They provide a flexible geometric model that still retains a strong PDE structure. We refer to \cite{leandro2023geometry, leandro2024electrostatic, cederbaum2016uniqueness, chrusciel2005non, chrusciel2017non, tiarlos, kunduri2018, lucietti} for a systematic discussion and for several examples, including the charged Nariai models and Reissner--Nordstr\"om--de Sitter metrics.
 
 In our previous work with Leandro \cite{leandro2024electrostatic}, we considered the three-dimensional case and imposed a divergence-free Bach tensor condition. Under this hypothesis, and assuming that the electric field and gradient of the lapse function are linearly dependent, we proved that electrostatic manifolds are locally conformally flat and described their local warped product structure. 
 
 The aim of the present study is to investigate the four-dimensional counterpart of this picture under a different curvature hypothesis. Here we work in dimension four and replace the Bach condition from \cite{leandro2024electrostatic} with the harmonicity of one (anti-)self-dual part of the Weyl tensor. We show that this assumption, combined with the electrostatic equations, still forces local conformal flatness and a warped product structure, now with a three-dimensional constant curvature fiber.
 
 \medskip
 
 We begin by recalling the electrostatic system in arbitrary dimension $n\ge3$.
 
 \begin{definition}\label{def1}
 	Let $(M^n,g)$ be a Riemannian manifold and let $E$ be a vector field on $M$ (the electric field) and $f\in C^\infty(M)$ (the lapse function). We say that $(M^n,g,f,E)$ satisfies the \emph{electrostatic system with cosmological constant $\Lambda$} if
 	\begin{equation}\label{s1}
 		\begin{array}{rcl}
 			\nabla^2f&=&f\left(\mathrm{Ric}-\dfrac{2}{n-1}\Lambda g+2E^\flat\otimes E^\flat-\dfrac{2}{n-1}|E|^2g\right),\\\\
 			\Delta f&=&2f\left(\dfrac{n-2}{n-1}|E|^2-\dfrac{\Lambda}{n-1}\right),\\[0.2cm]
 			\mathrm{div}(E)&=&0,\qquad d(fE^\flat)=0.
 		\end{array}
 	\end{equation}
 	Here $\mathrm{Ric}$, $\nabla^2$, $\mathrm{div}$ and $\Delta$ denote, respectively, the Ricci tensor, Hessian, divergence and Laplacian with respect to $g$, and $E^\flat$ is the $g$-dual one-form associated with $E$.	The condition
 	$
 	d(fE^\flat)=0$
 	is the coordinate-free formulation of the Maxwell equation and corresponds to the vanishing of the curl of $fE$ in the three-dimensional setting.  The Riemannian manifold $(M^n,g)$ is called the \emph{spatial factor} of the corresponding static electrostatic spacetime.
 \end{definition}

 Throughout this paper we assume $f>0$ on $M$. When $M$ has boundary, it is standard to require $f^{-1}(0)=\partial M$ (see, e.g., \cite{chrusciel2005non, chrusciel2017non, tiarlos, kunduri2018}). Taking the trace of the first equation in \eqref{s1} and combining it with the Laplacian equation for $f$ yields a relation between the scalar curvature $R$, electric field and cosmological constant which does not depend on $n$:
 \begin{equation}\label{rrr}
 	R=2(|E|^2+\Lambda).
 \end{equation}
 
 It is also well known that the electrostatic system forces a strong alignment between the electric field and the lapse function: the equations imply that $E$ and $\nabla f$ are linearly dependent along the boundary $\partial M=f^{-1}(0)$ (see \cite[Lemma 4]{tiarlos}). Motivated by this and by the explicit models, we shall assume that this dependence holds on the whole manifold.
 
 More precisely, in what follows we work under the hypothesis that there exists a smooth function $\rho:M\to\mathbb{R}$ such that
 \[
 E=\rho\,\nabla f.
 \]

 As in \cite{leandro2024electrostatic}, once local conformal flatness is known, one can describe the geometry of the electrostatic manifold more precisely. In our four-dimensional setting we obtain the following warped product structure.
 
 \begin{theorem}\label{thm:ricci-alignment}
 	Let $(M^4,g,f,E)$ be a four--dimensional electrostatic system such that the electric field and the gradient of the lapse function are linearly dependent. Assume that one of the (anti-)self-dual parts of the Weyl tensor is harmonic, i.e.,
 	\[
 	\delta W^{\pm}=0.
 	\]
 	Then, at every regular point of $f$, the gradient $\nabla f$ is an eigenvector of the Ricci operator. In particular, the regular level sets of $f$ are totally umbilic hypersurfaces with constant mean curvature.
 \end{theorem}
 
 Theorem~\ref{thm:ricci-alignment} shows that, in dimension four, the harmonicity of one of the self-dual or anti-self-dual components of the Weyl tensor plays a role analogous to the  divergence-free Bach tensor condition in dimension three.
 In both settings, this curvature hypothesis enforces an alignment between the Ricci tensor  and the gradient of the lapse function, leading to strong geometric rigidity of  electrostatic manifolds.

 \begin{theorem}\label{thm:warped-product}
 	Let $(M^4,g,f,E)$ be a four--dimensional electrostatic system such that the electric field and the gradient of the lapse function are linearly dependent. Suppose that the hypotheses of Theorem~\ref{thm:ricci-alignment} hold.  
 	If $f$ is a proper function, then in a neighborhood of any regular point of $f$ the metric $g$ is locally conformally flat and $(M^4,g)$ admits a local warped product decomposition
 	\[
 	(M^4,g)\simeq (I,dr^2)\times_{\varphi}(N^3,g_N),
 	\]
 	where $I\subset\mathbb{R}$ is an interval, $\varphi:I\to(0,\infty)$ is a smooth warping function, and $(N^3,g_N)$ is a three--dimensional Riemannian manifold of constant curvature.
 \end{theorem}
 If $M^4$ is compact, the assumption that $f$ be proper can be removed.
 
 This result provides a four-dimensional rigidity theorem under a natural
 curvature condition on the Weyl tensor, extending previous results obtained
 under Bach-flat assumptions. The proof is based on an integrability condition relating the Cotton tensor and the self-dual part of the Weyl tensor. This strategy originates in the work of Cao and Chen \cite{cao} on gradient Ricci solitons and has been successfully adapted to several geometric settings by Catino, Mastrolia, Monticelli and collaborators \cite{catino-mastrolia-monticelli-punzo,catino2017gradient}. In the present work we derive the corresponding identity for electrostatic systems with cosmological constant by introducing a tensor naturally associated with the Einstein–Maxwell equations.

 \section{Preliminaries and Integrability Identities}\label{lemmas}
 
 In this section we derive the fundamental identities that will be used throughout the proof of the main results. In particular, we recall the construction of a covariant $V$-tensor adapted to the electrostatic structure, following \cite{leandro2023geometry, leandro2024electrostatic} and the four-dimensional setting of \cite{barros2015critical, catino2012note, neto2016generalized}.
 
 Throughout this section $(M^4,g,f,E)$ denotes a four-dimensional electrostatic system with $E=\rho\nabla f$ as above.
 
 Combining \eqref{s1} with \eqref{rrr} in dimension $n=4$ gives
 \begin{equation*}
 	\nabla^2f=f\left(\mathrm{Ric}+2E^\flat\otimes E^\flat -\dfrac{R}{3}g\right).
 \end{equation*}
 To fix the notation, let $\{e_i\}_{i=1}^4$ be a local orthonormal frame on $M$ (see \cite{leandro2024electrostatic}). Then, this equation can be written in local coordinates as
 \begin{equation}\label{combinado}
 	\nabla_i\nabla_kf
 	=
 	f\left(
 	R_{ik}
 	+2E^\flat_iE^\flat_k
 	-\frac{R}{3}g_{ik}
 	\right),
 \end{equation}
 where $E^\flat_i:=E^\flat(e_i)$.
 
 Note that $d(fE^\flat)=0$ implies $$f\left(\nabla_iE^\flat_j-\nabla_jE^\flat_i\right)
 =
 \nabla_jf\,E^\flat_i
 -
 \nabla_if\,E^\flat_j,$$ taking the covariant derivative of \eqref{combinado} and using the Ricci identity yields
\begin{align}
	R_{ijkl}\nabla^lf
	={}&\,\nabla_i\nabla_j\nabla_kf
	-\nabla_j\nabla_i\nabla_kf \nonumber\\
	={}&\,
	\left(R_{jk}\nabla_if-R_{ik}\nabla_jf\right)\nonumber\\
	&-\frac{R}{3}
	\left(\nabla_if\,g_{jk}
	-\nabla_jf\,g_{ik}\right)\nonumber\\
	&+f\left(\nabla_iR_{jk}
	-\nabla_jR_{ik}\right)\nonumber\\
	&+2f\left(
	E^\flat_j\nabla_iE^\flat_k
	-E^\flat_i\nabla_jE^\flat_k
	\right)\nonumber\\
	&-\frac{f}{3}
	\left(\nabla_iR\,g_{jk}
	-\nabla_jR\,g_{ik}\right).\nonumber
\end{align}
 
 The Cotton tensor on a four-dimensional Riemannian manifold is defined by
 \begin{equation}\label{ct}
 	C_{ijk}=\nabla_iR_{jk}-\nabla_jR_{ik}-\frac{1}{6}(\nabla_iRg_{jk}-\nabla_jRg_{ik}).
 \end{equation}
 
 Combining the Ricci commutation formula with \eqref{ct}, we obtain the identity
 
 \begin{align*}
 	R_{ijkl}\nabla^lf
 	={}&\,fC_{ijk}
 	+\left(R_{jk}\nabla_if-R_{ik}\nabla_jf\right)\nonumber\\
 	&-\frac{R}{3}
 	\left(\nabla_if\,g_{jk}
 	-\nabla_jf\,g_{ik}\right)\nonumber\\
 	&+2f\left(
 	E^\flat_j\nabla_iE^\flat_k
 	-E^\flat_i\nabla_jE^\flat_k
 	\right)\nonumber\\
 	&-\frac{f}{6}
 	\left(
 	\nabla_iR\,g_{jk}
 	-\nabla_jR\,g_{ik}
 	\right).
 \end{align*}

On the other hand, the Riemann curvature tensor admits the Weyl $W$ decomposition (see, e.g., \cite[Chapter 1]{besse})
 \[
 \begin{aligned}
 	R_{ijkl}\nabla^{l}f
 	=
 	W_{ijkl}\nabla^{l}f
 	&+\frac{1}{2}\left(R_{ik}\nabla_{j}f-R_{jk}\nabla_{i}f\right)\\
 	&-\frac{1}{2}\left(R_{il}\nabla^{l}f\,g_{jk}
 	-R_{jl}\nabla^{l}f\,g_{ik}\right)\\
 	&-\frac{R}{6}\left(\nabla_{j}f\,g_{ik}
 	-\nabla_{i}f\,g_{jk}\right).
 \end{aligned}
 \]
 Equating the two expressions for $R_{ijkl}\nabla^l f$, we obtain
 \begin{equation}\label{construcao1}
 	\begin{split}
 		fC_{ijk}
 		={}&\,
 		W_{ijkl}\nabla^{l}f + \frac{3}{2}
 		\left(
 		R_{ik}\nabla_jf
 		-
 		R_{jk}\nabla_if
 		\right)\\
 		&-\frac12
 		\left(
 		R_{il}\nabla^{l}f\,g_{jk}
 		-
 		R_{jl}\nabla^{l}f\,g_{ik}
 		\right)\\
 		&+\frac{R}{2}
 		\left(
 		\nabla_if\,g_{jk}
 		-
 		\nabla_jf\,g_{ik}
 		\right)\\
 		&-2f
 		\left(
 		E^\flat_j\nabla_iE^\flat_k
 		-
 		E^\flat_i\nabla_jE^\flat_k
 		\right)\\
 		&+\frac{f}{6}
 		\left(
 		\nabla_iR\,g_{jk}
 		-
 		\nabla_jR\,g_{ik}
 		\right).
 	\end{split}
 \end{equation}

 Motivated by the integrability conditions introduced for gradient Ricci solitons, we define the following covariant $3-$tensor $V$

\begin{equation}\label{tt}
	\begin{split}
		V_{ijk}
		={}&\, \frac{3}{2}
		\left(
		R_{ik}\nabla_jf
		-
		R_{jk}\nabla_if
		\right)\\
		&-\frac12
		\left(
		R_{il}\nabla^{l}f\,g_{jk}
		-
		R_{jl}\nabla^{l}f\,g_{ik}
		\right)\\
		&+\frac{R}{2}
		\left(
		\nabla_if\,g_{jk}
		-
		\nabla_jf\,g_{ik}
		\right)\\
		&-2f
		\left(
		E^\flat_j\nabla_iE^\flat_k
		-
		E^\flat_i\nabla_jE^\flat_k
		\right)\\
		&+\frac{f}{6}
		\left(
		\nabla_iR\,g_{jk}
		-
		\nabla_jR\,g_{ik}
		\right).
	\end{split}
\end{equation}

 The tensor $V$ is clearly skew-symmetric in its first two indices. Moreover, using $\mathrm{div}(E)=0$ together with the fact $R=2(|E|^2+\Lambda)$, one verifies that $V$ is trace-free. Hence $V$ has the same algebraic symmetries as the Cotton tensor. Consequently, identity \eqref{construcao1} takes the form
 
 \begin{lemma}\label{lemma1}
 	Let $(M^4,\,g,\,f,\,E)$ be a four-dimensional electrostatic system. Then 
 	\begin{equation}\label{ttt3}
 		fC_{ijk}=W_{ijkl}\nabla^lf+V_{ijk}.
 	\end{equation}
 \end{lemma}
 
 \begin{lemma}\label{lemma2}
 	Let $(M^4,\,g,\,f,\,E)$ be a four-dimensional electrostatic system such that $E=\rho\nabla f$ for some smooth function $\rho$. Then the $V-$tensor is given by
 \begin{equation}\label{simpcom}
 	\begin{split}
 		V_{ijk}
 		={}&\,
 		\frac{3-4f^2\rho^2}{2}
 		\left(
 		R_{ik}\nabla_jf
 		-
 		R_{jk}\nabla_if
 		\right)\\
 		&-\frac{3-4f^2\rho^2}{6}
 		\left(
 		R_{il}\nabla^lf\,g_{jk}
 		-
 		R_{jl}\nabla^lf\,g_{ik}
 		\right)\\
 		&+\left(
 		\frac{3-4f^2\rho^2}{6}R
 		\right)
 		\left(
 		\nabla_if\,g_{jk}
 		-
 		\nabla_jf\,g_{ik}
 		\right).
 	\end{split}
 \end{equation}
Moreover, Lemma \ref{lemma1} also holds for the new tensor $V$.
 \end{lemma}
 \begin{proof}
 	Since the electric field and the lapse function are linearly dependent, there exists a smooth function $\rho$ such that $$E = \rho\nabla f.$$ Differentiating $E=\rho\nabla f$ and using \eqref{combinado}, we get
 	\[
 	\begin{aligned}
 		\nabla_iE^\flat_k
 		&=\nabla_i(\rho\nabla_kf)\\
 		&=\nabla_i\rho\,\nabla_kf
 		+\rho\nabla_i\nabla_kf\\
 		&=\nabla_i\rho\,\nabla_kf
 		+\rho f\left(
 		R_{ik}
 		+2\rho^2\nabla_if\nabla_kf
 		-\frac{R}{3}g_{ik}
 		\right)\\
 		&=\nabla_i\rho\,\nabla_kf
 		+f\rho R_{ik}
 		+2f\rho^3\nabla_if\nabla_kf
 		-\frac{f\rho R}{3}g_{ik}.
 	\end{aligned}
 	\]

 	We next compute the gradient of the scalar curvature. Differentiating \eqref{rrr} and using $E=\rho\nabla f$ we obtain
 	 \[
 	 \nabla_i R=2\nabla_i|E|^2
 	 =4\rho|\nabla f|^2\nabla_i\rho+2\rho^2\nabla_i|\nabla f|^2.
 	 \]
 	 From \eqref{combinado},
 	 \[
 	 \nabla_i|\nabla f|^2
 	 =2f\left(R_{il}\nabla^l f+2\rho^2|\nabla f|^2\nabla_i f-\frac{R}{3}\nabla_i f\right).
 	 \]
 	 Combining these expressions and using \eqref{rrr}, we find
 	 
 	 \[
 	 \begin{aligned}
 	 	\nabla_iR
 	 	={}&\,4\rho|\nabla f|^2\nabla_i\rho
 	 	+4f\rho^2R_{il}\nabla^lf\\
 	 	&+8f\rho^4|\nabla f|^2\nabla_if
 	 	-\frac{4f\rho^2R}{3}\nabla_if.
 	\end{aligned}
 	 \]

 Substituting the above identities into \eqref{tt} we obtain
 	 
 \begin{equation}\label{simp}
 	\begin{split}
 		V_{ijk}
 		={}&\,
 		\frac{3-4f^2\rho^2}{2}
 		\left(
 		R_{ik}\nabla_jf
 		-
 		R_{jk}\nabla_if
 		\right)\\
 		&-\frac{3-4f^2\rho^2}{6}
 		\left(
 		R_{il}\nabla^lf\,g_{jk}
 		-
 		R_{jl}\nabla^lf\,g_{ik}
 		\right)\\
 		&+\frac{2f\rho|\nabla f|^2}{3}
 		\left(
 		\nabla_i\rho\,g_{jk}
 		-
 		\nabla_j\rho\,g_{ik}
 		\right)\\
 		&-2f\rho
 		\left(
 		\nabla_i\rho\,\nabla_jf\,\nabla_kf
 		-
 		\nabla_j\rho\,\nabla_if\,\nabla_kf
 		\right)\\
 		&+\left(
 		\frac{9-16f^2\rho^2}{18}R
 		+\frac{4f^2\rho^4|\nabla f|^2}{3}
 		\right)
 		\left(
 		\nabla_if\,g_{jk}
 		-
 		\nabla_jf\,g_{ik}
 		\right).
 	\end{split}
 \end{equation}
 
 Now, considering that $\mathrm{div}(E)=0$ \eqref{s1}, we have
 \[
 0=\operatorname{div}(\rho\nabla f)
 =\nabla_i(\rho\nabla^if)
 =\nabla_i\rho\,\nabla^if+\rho\,\Delta f.
 \]
 Hence,
 \[
 \langle\nabla\rho,\nabla f\rangle
 =-\rho\,\Delta f.
 \]
 
 Moreover, since
 \[
 d(fE^\flat)=0,
 \]
 it follows that
 \[
 \nabla_i\rho\,\nabla_jf
 -
 \nabla_j\rho\,\nabla_if=0,
 \]
 so that \(\nabla\rho\) is parallel to \(\nabla f\) (see \cite{leandro2024electrostatic}). Therefore, there exists a smooth function \(\mu\) such that
 \[
 \nabla\rho=\mu\nabla f.
 \]
 Taking the inner product with \(\nabla f\), we obtain
 \[
 \mu|\nabla f|^2
 =
 -\rho\,\Delta f,
 \]
 that is,
 \[
 |\nabla f|^2\nabla\rho
 =
 -\rho\,\Delta f\nabla f.
 \]
 Using the Laplacian equation for $f$ of \eqref{s1}, we get that
 \[
 |\nabla f|^2\nabla\rho
 =
 -\frac{2\rho f}{3}
 \left(
 2\rho^2|\nabla f|^2-\Lambda
 \right)\nabla f,
 \]

 Finally, using $E=\rho\nabla f$ in \eqref{rrr}, we obtain
 
 \[
 |\nabla f|^2\nabla\rho
 =
 -\frac{2\rho f}{3}
 \left(
 3\rho^2|\nabla f|^2-\dfrac{R}{2}
 \right)\nabla f,
 \]
Therefore, substituting these identities into \eqref{simp} yields 

\begin{equation*}
	\begin{split}
		V_{ijk}
		={}&\,
		\frac{3-4f^2\rho^2}{2}
		\left(
		R_{ik}\nabla_jf
		-
		R_{jk}\nabla_if
		\right)\\
		&-\frac{3-4f^2\rho^2}{6}
		\left(
		R_{il}\nabla^lf\,g_{jk}
		-
		R_{jl}\nabla^lf\,g_{ik}
		\right)\\
		&+\left(
		\frac{3-4f^2\rho^2}{6}R
		\right)
		\left(
		\nabla_if\,g_{jk}
		-
		\nabla_jf\,g_{ik}
		\right).
	\end{split}
\end{equation*}

A straightforward computation completes the proof.

 \end{proof}
 
Motivated by \cite{leandro2024electrostatic} and by the coefficients appearing in the expression of the tensor $V$ when $E=\rho\nabla f$, we establish the following lemma.
 
 \begin{lemma}\label{lemmabom}
 	The set $\{3-4f^2\rho^2\neq0\}$ is dense in $(M^4,\,g,\, f,\, E)$, in which $E = \rho\nabla f$.
 \end{lemma}
 \begin{proof}
 	Assume, for contradiction, that there exists a nonempty open subset
 	$\Omega\subset M$ on which
 	\[
 	3-4f^2\rho^2=0.
 	\]
 	
 	Differentiating the above identity, we obtain
 	\[
 	f\rho^2\nabla f+f^2\rho\nabla\rho=0.
 	\]
 	Taking the inner product with $\nabla f$ and using
 	$|E|^2=\rho^2|\nabla f|^2$, we obtain
 	\[
 	|E|^2+f\rho\langle\nabla\rho,\nabla f\rangle=0.
 	\]
 	
 	On the other hand, by \eqref{s1},
 	\[
 	0=\operatorname{div}(E)
 	=\rho\Delta f+\langle\nabla\rho,\nabla f\rangle
 	=\frac{2}{3}\rho f(2|E|^2-\Lambda)
 	+\langle\nabla\rho,\nabla f\rangle.
 	\]
 	Hence,
 	\[
 	f\rho\langle\nabla\rho,\nabla f\rangle
 	=\frac{2}{3}\rho^2f^2(\Lambda-2|E|^2)
 	=\frac{1}{2}(\Lambda-2|E|^2).
 	\]
 	
 	Comparing the last two identities, we obtain
 	\[
 	|E|^2=\frac{1}{2}(2|E|^2-\Lambda),
 	\]
 	which implies
 	\[
 	\Lambda=0.
 	\]
 	This contradicts the standing assumption that $\Lambda\neq0$.
 	Therefore, the set
 	\[
 	\{3-4f^2\rho^2\neq0\}
 	\]
 	is dense in \(M\).
 \end{proof}
 
 The next lemma expresses the divergence of the self-dual Weyl tensor in terms of the Cotton tensor. This identity is classical and goes back to \cite{szekeres1968conformal}; see also \cite{derdzinski1983} and \cite{besse}. It will be used to bring the hypothesis $\delta W^+=0$ into the integral identities below. Since our arguments rely on the self-dual decomposition of the Weyl tensor, we briefly recall the necessary notation. Let $(M^4,\,g)$ be an oriented Riemannian manifold. The Hodge star operator $$*:\Lambda^2(M)\rightarrow \Lambda^2(M)$$
 satisfies $*^2=Id$, yielding the orthogonal decomposition $$\Lambda^2=\Lambda^2_+\oplus \Lambda_-^2,$$
 
 where $$\Lambda_{\pm}^2 =\{\omega\in\Lambda^2:*\omega=\pm\omega\}.$$
 
 The Weyl tensor acts as a symmetric endomorphism of $\Lambda^2$
 and preserves this splitting, giving rise to the decomposition 
 $$W=W^+\oplus W^-.$$
 
 Throughout the paper we adopt the barred-index notation introduced by \cite[Section 3]{szekeres1968conformal} to denote the Hodge-dual directions with respect to an oriented orthonormal frame. More precisely, with respect to an oriented orthonormal frame $\{e_1,\,e_2,\,e_3,\,e_4\}$,
$$\bar 1\bar 2=34, \bar1 \bar3=42, \bar1 \bar4=23,$$
 and similarly for the remaining pairs according to the chosen orientation.
 
 \begin{lemma}\label{lemma-div-W+}
 	Let $(M^4,g)$ be an oriented Riemannian $4$--manifold and let $W=W^+ \oplus W^-$ be the decomposition of the Weyl tensor into its self-dual and anti-self-dual parts. In a local oriented orthonormal frame $\{e_1,\dots,e_4\}$ adapted to the splitting of $\Lambda^2$, we have
 	\begin{equation}\label{div-Wplus-Cotton}
 		4\nabla^i W^+_{ijkl} \;=\; C_{klj} + C_{\bar k \bar l j},
 	\end{equation}
 	where the indices with a bar denote the Hodge-dual directions. In particular, if $\delta W^+=0$, then
 	\begin{equation}\label{C-plus-relation}
 		C_{klj} + C_{\bar k \bar l j} = 0.
 	\end{equation}
 \end{lemma}
 
 \begin{proof}
 	The Cotton tensor and the Weyl tensor are related by the standard identity
 	\begin{equation}\label{cw}
 		C_{ijk}=2\nabla^lW_{lkij}.
 	\end{equation}  In an oriented orthonormal frame adapted to the splitting of $\Lambda^2$, the Weyl tensor decomposes as $W=W^+ \oplus W^-$. The following identity is well known (see \cite{derdzinski1983}, \cite[Chapter 16]{besse} or \cite{barros2015critical}) shows that
 	\[
 	\nabla^i W^+_{ijkl} \;=\; \frac{1}{4} \big( \nabla^i W_{ijkl} + \nabla^i W_{\bar i \bar j kl} \big).
 	\]
 	Using \eqref{cw} and the symmetries of $W$, one checks that
 	\[
 	4\nabla^i W^+_{ijkl} \;=\; C_{klj} + C_{\bar k \bar l j},
 	\]
 	which is \eqref{div-Wplus-Cotton}. If $\delta W^+=0$, then the left-hand side vanishes and we obtain \eqref{C-plus-relation}.
 \end{proof}

 \section{Proof of the main results}\label{proofs}
 
 The key point is that the harmonicity of one of the self-dual components allows us to convert an analytic condition on the Weyl tensor into an algebraic condition on the Ricci tensor through Lemmas $\ref{lemma1}-\ref{lemma-div-W+}$. We start by extracting algebraic information from the harmonicity of $W^\pm$, following the approach in \cite{barros2015critical, cao, sun, catino2017gradient, neto2016generalized}. Using Lemma~\ref{lemma-div-W+}, we have in an oriented orthonormal frame adapted to the splitting of $\Lambda^2$:
 \[
 4\nabla^iW^+_{ijkl}=C_{klj}+C_{\bar{k}\bar{l}j},
 \]
 where the bar indices correspond to Hodge-dual directions. If $\delta W^+=0$, then
 \[
 C_{klj}+C_{\bar{k}\bar{l}j}=0.
 \]
 Using \eqref{ttt3} this can be written as
 \[
 V_{\bar{i}\bar{j}k}+V_{ijk}=-(W_{\bar{i}\bar{j}kp}+W_{ijkp})\nabla^pf.
 \]
 By the symmetries of the Weyl tensor, this implies
 \begin{equation}\label{semW}
 	(V_{\bar{i}\bar{j}k}+V_{ijk})\nabla^kf=0.
 \end{equation}

 Contracting \eqref{simpcom} with $\nabla^kf$, we obtain
 \[
 V_{ijk}\nabla^kf
 =
 \frac{3-4f^2\rho^2}{3}
 \left(
 R_{ik}\nabla^kf\,\nabla_jf
 -
 R_{jk}\nabla^kf\,\nabla_if
 \right).
 \]

 By Lemma \ref{lemmabom}, $$\{3-4f^2\rho^2\neq0\}$$ is dense in $M$. Therefore, plugging the above equation into \eqref{semW}, we obtain
 \[
 (R_{\bar{j}k}\nabla^kf\nabla_{\bar{i}}f-R_{\bar{i}k}\nabla^kf\nabla_{\bar{j}}f)+(R_{jk}\nabla^kf\nabla_if-R_{ik}\nabla^kf\nabla_jf)=0.
 \]
 
 Recalling that the bar denotes the Hodge-dual directions, we fix a point $q\in\Sigma=f^{-1}(c)$ where $\nabla f(q)\neq0$, and let $\{e_1,e_2,e_3,e_4\}$ be an orthonormal frame diagonalizing the Ricci tensor at $q$, i.e., $\mathrm{Ric}_q(e_i,e_j)=R_{ii}(q)\delta_{ij}$. In this frame, $R_{jk}\nabla^kf=R_{jj}\nabla^jf$, then the previous equality leads to the system 
 \begin{equation}\label{r11}
 	\left\{\begin{array}{c}
 		(R_{11} -R_{22})\nabla_1f\nabla_2f +(R_{33} -R_{44})\nabla_3f\nabla_4f =0,\\[0.2cm]
 		(R_{11} -R_{33})\nabla_1f\nabla_3f +(R_{44} -R_{22})\nabla_4f\nabla_2f =0,\\[0.2cm]
 		(R_{11} -R_{44})\nabla_1f\nabla_4f +(R_{22} -R_{33})\nabla_2f\nabla_3f =0.
 		
 	\end{array}\right.
 \end{equation}
 
 We are now in a position to prove that $\nabla f$ is an eigenvector of the Ricci tensor. If $\nabla f(p)\neq0$ in a single direction, say $\nabla_1f\neq0$ and $\nabla_2f=\nabla_3f=\nabla_4f=0$, then $\mathrm{Ric}(\nabla f)=R_{11}\nabla f$. If $\nabla f(p)$ has exactly two non-zero components, for example $\nabla_1f\neq0,\nabla_2f\neq0$ and the others vanish, then \eqref{r11} implies $R_{11}=R_{22}$, and again $\mathrm{Ric}(\nabla f)$ is proportional to $\nabla f$. A similar argument works if three components are non-zero. Finally, if $\nabla_if\neq0$ for all $i$, then squaring and summing the equations in \eqref{r11} yields
 \begin{eqnarray*}
 	&&(R_{11} -R_{22})^2(\nabla_1f\nabla_2f)^2 +(R_{33} -R_{44})^2(\nabla_3f\nabla_4f)^2\\
 	&&+(R_{11} -R_{33})^2(\nabla_1f\nabla_3f)^2 +(R_{44} -R_{22})^2(\nabla_4f\nabla_2f)^2\\
 	&&+(R_{11} -R_{44})^2(\nabla_1f\nabla_4f)^2 +(R_{22} -R_{33})^2(\nabla_2f\nabla_3f)^2 =0,
 \end{eqnarray*}
 hence $R_{11}=R_{22}=R_{33}=R_{44}$. In all cases, $\nabla f$ is an eigenvector of $\mathrm{Ric}$.

 \subsection{Warped product structure and proof of Theorem \ref{thm:warped-product}}
 
 We now turn to the geometric structure of the electrostatic manifold. We already know that $\nabla f$ is an eigenvector of the Ricci tensor on the regular set of $f$. In particular, on each connected component of the regular set there exists an orthonormal frame $\{e_1,e_2,e_3,e_4\}$ such that
 \[
 e_1=\frac{\nabla f}{|\nabla f|},\quad e_2,e_3,e_4\in T\Sigma,\quad \mathrm{Ric}=\mathrm{diag}(R_{11},R_{22},R_{33},R_{44}).
 \]
 
 From \eqref{combinado} we obtain, for $a\in\{2,3,4\}$,
 \begin{eqnarray*}
 	\nabla_a|\nabla f|^2=2f\left(R_{al}\nabla^l f+2\rho^2|\nabla f|^2\nabla_a f-\frac{R}{3}\nabla_a f\right),
 \end{eqnarray*}
 so $|\nabla f|$ is constant along each level set $\Sigma=f^{-1}(c)$. Locally we can therefore write the metric as
 \[
 g = \frac{1}{|\nabla f|^{2}}df^{2} + g_{ab}(f,\theta)\,d\theta_{a}d\theta_{b},
 \]
 where $(\theta_2,\theta_3,\theta_4)$ are coordinates on the level sets. Since $f$ is analytic, the regular set of $f$ is open and dense; we work on a connected component of a neighborhood $U_I=f^{-1}(I)$ of a regular level, where $I\subset\mathbb R$ is an interval.
 
 Setting
 \[
 r(x)=\int\frac{df}{|\nabla f|},
 \]
 we obtain local coordinates $(r,\theta_2,\theta_3,\theta_4)$ on $U_I$ in which
 \[
 g = dr^{2}+g_{ab}(r,\theta)d\theta_{a}d\theta_{b},\qquad \nabla f=f'(r)\,\partial_r.
 \]
 
 The second fundamental form of the level sets $\Sigma_r=\{r=\mathrm{const}\}$ is
 \begin{eqnarray}\label{eq555}
 	h_{ab}&=& - \langle e_{1},\,\nabla_{a}e_{b}\rangle=\frac{1}{|\nabla f|} \nabla_a \nabla_b f\\
 	&=&\frac{f}{|\nabla f|} \left( R_{ab} - \frac{R}{3} g_{ab} \right)
 	= \frac{f}{|\nabla f|} \left( \mu - \frac{R}{3} \right) g_{ab},\nonumber
 \end{eqnarray}
 
 where $\mu$ is the eigenvalue of $\mathrm{Ric}$ in directions tangent to $\Sigma_r$. Since $h_{ab}$ is proportional to $g_{ab}$, we can write
 \[
 h_{ab}=\frac{H}{3}g_{ab},
 \]
 where $H=H(r)$ is the mean curvature of the level sets. Thus, the level sets are totally umbilic.
 
 Contracting the Codazzi equation 
 \begin{eqnarray*}
 	R_{1cab}=\nabla_{a}h_{bc}-\nabla_{b}h_{ac}
 \end{eqnarray*}
 over $b$ and $c$ we obtain
 \[
 R_{1a}=\nabla_{a}(H)-\frac{1}{3}\nabla_{a}(H)=\frac{2}{3}\nabla_{a}(H).
 \]
 Since $R_{1a}=0$, it follows that $H$ is constant along each level set.
 
 We now write the metric near a regular level in adapted coordinates. Let
 $$(x_{1},\, x_2,\, x_{3},\,x_4) = (r,\,\theta_2,\, \theta_{3},\,\theta_4), $$
 where $(\theta_2,\theta_3,\theta_4)$ are local coordinates on $\Sigma_r$. Using $a,b\in\{2,3,4\}$, we have
 \begin{eqnarray*}
 	h_{ab}=-g(\partial_r,\, \nabla_{a}\partial_{b})=-g(\partial_r, \Gamma^{1}_{ab}\partial_{r})=-\Gamma^{1}_{ab}.
 \end{eqnarray*}
 Since $g_{11}=1$, we obtain
 \[
 \Gamma^{1}_{ab}=\frac{1}{2}g^{11}\left(-\frac{\partial}{\partial r}g_{ab}\right)=\frac{-1}{2}\frac{\partial}{\partial r}g_{ab},
 \]
 and hence
 \[
 \frac{\partial}{\partial r}g_{ab}=  H(r)\,g_{ab}.
 \]
 Integrating in $r$,
 \[
 g_{ab}(r,\theta)=\varphi(r)^{2}g_{ab}(r_{0},\theta),
 \]
 where $\varphi(r)=\exp\!\big(\int^{r}_{r_{0}}H(s)\,ds\big)$ and $\{r=r_{0}\}$ corresponds to a fixed level set $\Sigma_{r_0}$. Thus, on $U_I$ the metric can be written in warped product form
 \[
 (M^{4},g)=(I,dr^2)\times_\varphi (N^3,\overline g),
 \]
 with $g=dr^2+\varphi^2\overline g$, where $(N^3,\overline g)$ is a level set endowed with the induced metric.
 
 The Ricci tensor of $(M^4,g)$ can be written (see \cite{besse}) as
 \begin{eqnarray}\label{wps}
 	R_{11}=-3\frac{\varphi''}{\varphi},\qquad R_{1a}=0,
 \end{eqnarray}
 and
 \begin{eqnarray*}
 	R_{ab}=\overline{R}_{ab}-\left[2(\varphi')^2+\varphi\varphi''\right]\overline{g}_{ab},
 \end{eqnarray*}
 where $\overline{R}_{ab}$ and $\overline{g}_{ab}$ denote the Ricci tensor and the metric of $(N^3,\overline g)$, respectively. If $(N^3,\overline g)$ has constant scalar curvature $\overline R$, then
 \begin{eqnarray*}
 	\overline{R}_{ab}=\frac{\overline{R}}{3}\,\overline{g}_{ab},
 \end{eqnarray*}
 and
 \begin{eqnarray*}
 	R_{ab}=\left[\frac{\overline{R}}{3}-2(\varphi')^2-\varphi\varphi''\right]\overline{g}_{ab}.
 \end{eqnarray*}
 The scalar curvature of $g$ is
 \[
 R=\varphi^{-2}\overline{R}-6\left(\frac{\varphi'}{\varphi}\right)^2-6\frac{\varphi''}{\varphi},
 \]
 so that
 \begin{equation}\label{Rbar=R}
 	\overline{R} = \varphi^2R+6(\varphi')^2+6\varphi\varphi''.
 \end{equation}
 Using \eqref{rrr} and $|\nabla f|^2=(f')^2$, we can rewrite this as
 \begin{equation}\label{barR}
 	\overline{R} = 2\varphi^2\rho^2(f')^2 + 2\varphi^2\Lambda + 6(\varphi')^2 + 6\varphi\varphi''.
 \end{equation}

 Moreover, from \eqref{combinado} we have
 \[
 \frac{1}{|\nabla f|^2}\langle\nabla|\nabla f|^2,\,\nabla f\rangle
 =2f\left(R_{11}+2\rho^2|\nabla f|^2-\frac{R}{3}\right).
 \]
 Using \eqref{rrr} and \eqref{wps}, and $|\nabla f|^2=(f')^2$, we obtain
 \begin{eqnarray*}
 	\langle\nabla|\nabla f|^2,\,\nabla f\rangle
 	&=&2f(f')^2\left[
 	R_{11}+2\rho^2(f')^2-\frac{R}{3}
 	\right]\\
 	=2f(f')^2&&\left[
 	-3\frac{\varphi''}{\varphi}
 	+2\rho^2(f')^2-\frac{2}{3}\big(\rho^2(f')^2+\Lambda\big)
 	\right]\\
 	&=&2f(f')^2\left[
 	-3\frac{\varphi''}{\varphi}
 	+\frac{4}{3}\rho^2(f')^2-\frac{2}{3}\Lambda
 	\right].
 \end{eqnarray*}
 Since $\nabla f = f'\partial_r$,
 \[
 \langle\nabla|\nabla f|^2,\,\nabla f\rangle
 =2(f')^2f'',
 \]
 hence
 \[
 2(f')^2f''=2f(f')^2\left[
 -3\frac{\varphi''}{\varphi}
 +\frac{4}{3}\rho^2(f')^2-\frac{2}{3}\Lambda
 \right].
 \]
 Dividing by $2(f')^2$ we obtain
 \begin{eqnarray}\label{rho2-4d}
 	f''&=&f\left[
 	-3\frac{\varphi''}{\varphi}
 	+\frac{4}{3}\rho^2(f')^2-\frac{2}{3}\Lambda
 	\right],\\
 	\rho^2&=&\frac{3}{4(f')^2}\left(
 	\frac{f''}{f}+3\frac{\varphi''}{\varphi}+\frac{2}{3}\Lambda
 	\right).\nonumber
 \end{eqnarray}
 
 Combining \eqref{rho2-4d} with \eqref{barR}, we see that $\overline{R}$ does not depend on the coordinates $\theta$. Therefore $\overline{R}$ is a constant, and $(N^3,\overline g)$ has constant scalar curvature. This proves the curvature part of Theorem~\ref{thm:warped-product}.
 
 Finally, we relate the mean curvature $H$ of the level sets to the warping function $\varphi$. Along a regular level set of $f$ (where $\nabla_af=0$), \eqref{eq555} and \eqref{combinado} give
 \begin{eqnarray*}
 	h_{ab}
 	=\frac{\nabla_a\nabla_bf}{|\nabla f|}
 	=\frac{f}{|\nabla f|}\left(R_{ab}-\frac{R}{3}g_{ab}\right).
 \end{eqnarray*}
 On the other hand, in warped product coordinates,
 \[
 h_{ab}=-\,\frac{\varphi'}{\varphi}\,g_{ab}.
 \]
 Comparing and using the formulas for $R_{ab}$ and $R$ above, one checks that
 \[
 R_{ab}-\frac{R}{3}g_{ab}
 =\varphi\varphi''\,\overline{g}_{ab},
 \]
 so that
 \[
 \frac{f}{|\nabla f|}\,\varphi\varphi''\,\overline{g}_{ab}
 =-\,\frac{\varphi'}{\varphi}\,\varphi^2\overline{g}_{ab}.
 \]
 Cancelling common factors and using $|\nabla f|=f'$, we obtain the ODE
 \begin{equation}\label{phi-ode}
 	f\,\varphi^2\varphi''+f'\,\varphi'=0.
 \end{equation}
 This shows that $\varphi$ is determined (up to multiplicative and additive constants) by the lapse function $f$. Together with the previous discussion, this completes the proof of Theorem~\ref{thm:warped-product}.
 
 \subsection{Examples}
 
 In this subsection we briefly discuss classical electrostatic solutions that fit
 into the framework of Theorem \ref{thm:warped-product} and illustrate the rigidity results obtained above. For further details on these examples in dimension four, as well as on other related
 models, we refer the reader to \cite{leandro2023geometry, leandro2024electrostatic, tiarlos}
 and the references therein.
 
 \medskip
 
 \noindent
 \textbf{Example 3.1 (Reissner--Nordstr\"om--de Sitter).}
 The Reissner--Nordstr\"om--de Sitter spacetime provides a fundamental example of
 a static Einstein--Maxwell solution with cosmological constant. Its spatial slices
 $(M^4,g)$ arise as electrostatic systems, where the lapse function $f$ depends only
 on a radial coordinate and the electric field $E$ is everywhere collinear with
 $\nabla f$.
 
 In this setting, the spatial metric can be written locally as a warped product with
 one-dimensional base and three-dimensional fibers of constant sectional curvature.
 In particular, the regular level sets of the lapse function are totally
 umbilic with constant mean curvature, in agreement with the conclusions of
 Theorem \ref{thm:warped-product}. Moreover, these metrics are locally conformally flat, consistently with
 the vanishing of the Cotton tensor in dimension four.
 
 \medskip
 
 \noindent
 \textbf{Example 3.2 (Charged Nariai-type solutions).}
 Another family of examples is given by charged Nariai-type solutions, which arise
 as limits of Reissner--Nordstr\"om--de Sitter spacetimes when the black hole and
 cosmological horizons coincide. In this case, the spatial metric splits locally as a
 product (or warped product) of a one-dimensional factor and a three-dimensional
 space form.
 
 These solutions satisfy the electrostatic equations with cosmological constant,
 and the electric field is again aligned with the gradient of the lapse function.
 The geometry therefore falls within the scope of Theorem \ref{thm:warped-product}, providing explicit
 models where the warped product structure and the constant curvature of the fiber
 are realized.

 \subsection*{Declaration of generative AI and AI-assisted technologies in the manuscript preparation process}
 
 During the preparation of this work, the author used ChatGPT to assist with language revision and to help identify potential gaps in the exposition. All results, proofs, and conclusions were independently verified by the author, who takes full responsibility for the content of the manuscript.

\bibliographystyle{abbrv}

\bibliography{casrefs}



\end{document}